\documentclass[12pt, reqno]{amsart}
\usepackage{amsmath, amsthm, amscd, amsfonts, amssymb, graphicx, color}
\usepackage[bookmarksnumbered, colorlinks, plainpages]{hyperref}
\hypersetup{colorlinks=true,linkcolor=red, anchorcolor=green, citecolor=cyan, urlcolor=red, filecolor=magenta, pdftoolbar=true}

\newtheorem{theorem}{Theorem}[section]
\newtheorem{lemma}[theorem]{Lemma}

\newtheorem{corollary}[theorem]{Corollary}
\theoremstyle{definition}
\newtheorem{definition}[theorem]{Definition}

\theoremstyle{remark}

\numberwithin{equation}{section}

\begin{document}

\setcounter{page}{1}

\title[Short Title]{Estimates for $L^\varphi$-Lipschitz and $L^\varphi$-BMO Norms of Differential Forms }

\author[ Xuexin Li, Jinling Niu \MakeLowercase {and} Yuming Xing]{Xuexin Li $^2$, Jinling Niu $^1$ \MakeLowercase{and} Yuming Xing$^{1*}$}

\address{$^{1}$Department of Mathematics, Harbin Institute of Technology, 150001, Harbin, P.R.China}
\email{\textcolor[rgb]{0.00,0.00,0.84}{niujinling@hit.edu.cn}}
\email{\textcolor[rgb]{0.00,0.00,0.84}{xyuming@hit.edu.cn}}

\address{$^{2}$Department of Mathematics, Northeast Forestry University, 150040, Harbin,  P.R.China}
\email{\textcolor[rgb]{0.00,0.00,0.84}{li406469482@163.com}}



\date{Received: xxxxxx; Revised: yyyyyy; Accepted: zzzzzz.
\newline \indent $^{*}$Corresponding author}

\begin{abstract}
In this paper, we define the $L^\varphi$-Lipschitz norm and $L^\varphi$-BMO norm of differential forms using Young functions, and prove the comparison theorems for the homotopy operator $T$ on differential forms with $L^\varphi$-Lipschitz and $L^\varphi$-BMO norms. As applications, we give the $L^\varphi$-BMO norm estimate for  conjugate $A$-harmonic tensors and the weighted $L^\varphi$-Lipschitz norm estimate for the homotopy operator $T$.
\end{abstract} \maketitle

\section{Introduction}

\vspace{3mm}

The bounded mean oscillation space was originally introduced by John and Nirenberg in 1961, which played an important role in the investigation of the solutions of the elliptic partial differential equations. Afterwards, Fefferman and Stein found that the bounded mean oscillation space was the dual space of the Hardy space and demonstrated the Fefferman-Stein decomposition, which became the bond revealing the  intrinsic relationship between the bounded mean oscillation space and harmonic analysis. Therefore,  the study on the bounded mean oscillation space becomes an essential part of harmonic analysis. For example, if we want to obtain the boundedness of the operator $T:H^1\rightarrow L^1$, due to that the BMO space is the dual space of the $H^1$ space, we can consider the boundedness of the dual operator $T^\ast:L^\infty\rightarrow BMO.$  Another classical application is that many classical operators are bounded from $L^p$ space to $L^p$ space for $1<p<\infty$, however, when $p=\infty$, the boundedness may be not correct. Instead of, we may obtain the boundedness of operator from $L^\infty$ space to BMO space.

\vskip 8 pt

As the natural generalizations of the Lebesgue space $L^p$,  the Orlicz space was first studied by Orlicz. From then on, the theory of the Orlicz spaces has been extensively developed in the area of analysis. Meanwhile, it has wide applications in probability, statistics, potential theory, partial differential equations, for instance \cite{Rao2002Applications}. Recently, there have been intense research activities on regularity theory in Orlicz spaces connected to a Young function, which satisfies some moderate growth conditions for second-order elliptic and parabolic PDEs, see \cite{Jia2007Regularity}. Moreover, the Orlicz-Hardy spaces are also good substitutes of the Orlicz spaces in dealing with many problems in analysis, for example, the boundedness of operators. The study of its dual spaces Orlicz-BMO spaces can be traced to the work of Janson in 1980. He generalized the classical Hardy space and BMO space, and obtained the dual relationship. All theories of these spaces are closely connected with properties of harmonic analysis, and of the Laplacian operator on $\mathbb{R}^n$. In recent ten years, Ding \cite{DingPoincar,Ding2010}, Bi \cite{Bi2011Orlicz} and Yi \cite{6} discussed the properties of operators or composite operators with Orlicz norm acting on differential forms. In this paper, we will introduce two generalized spaces, called $L^{\varphi}$-BMO space and $L^{\varphi}$-Lipchitz space. The traditional BMO space and Lipchitz space can be taken as special cases of our two new  spaces, if we let the function $\varphi=t^{p}, 1<p<\infty$. Then, we will establish the $L^\varphi$-BMO norm estimates of the homotopy operator for differential forms. Especially, when the differential forms satisfy the conditions of the Weak Reverse H\"{o}lder class (in \cite{Johnson2013Integral}), we obtain the $L^\varphi$-Lipschitz norm estimates. We did it because it can be used to study the $L^\varphi$-BMO and $L^{\varphi}$-Lipchitz norm estimates of some complicated composition of operators, such as the composition $T \circ H$ of homotopy and
projection operators and the composition $T \circ G$ of homotopy and Green's operators. More results on the norm inequalities for differential forms and homotopy operator can be found in \cite{Liu2010Some,Ding2009A,Ding2015Norm,Bi2011Some,Ding2009Lipschitz,1}.

The main purpose of this paper is to estimate the $L^\varphi$-BMO norm and $L^\varphi$-Lipschitz norm for the  homotopy operator on differential forms. The paper is organised as follows. Section 2 contains, in addition to definitions and other preliminary material, the main lemmas.  Theorem \ref{d6} and  Theorem \ref{d11} in Section 2 show the estimates for  the  homotopy operator with the $L^\varphi$-BMO norm and $L^\varphi$-Lipschitz norm by $L^\varphi$ norm. The conditions for differential forms $u$ in the two theorems are different, especially, the similar estimate as Theorem \ref{d6} with the  condition in Theorem \ref{d11} has not been proved. Whereafter, the comparison for the $L^\varphi$-BMO norm and $L^\varphi$-Lipschitz norm are given. As applications,  we use the results and methods in the previous section to estimate the conjugate $A$-harmonic tensors in Section 4. In this section we also get a weighted estimate for differential forms.

\vskip 8 pt

\section{The Main Definitions and Lemmas}
 Before specifying the main results precisely, we introduce some notations. We write $\Omega$ for a bounded convex domain in $\mathbb{R}^n$, $n\geq2$, endowed with the usual Lebesgue measure denoted by $|\Omega|$. $B$ and $\sigma{B}$ are concentric balls with $\hbox{diam}(\sigma{B})=\sigma{\hbox{diam}(B)}$.
The set of $l\hbox{-forms}$, denoted by $\Lambda^l=\Lambda^l(\mathbb{R}^n)$, is a
$l\hbox{-vector}$, spanned by exterior products  $e_I={e_{i_1}}\wedge{e_{i_2}}\wedge \cdots \wedge{e_{i_l}}$, for all ordered $l\hbox{-tuples}$ $I=(i_1,i_2,\cdots,i_l)$, $1\leq{i_1}<{i_2}<\cdots<{i_l}\leq n$. The $l\hbox{-form}$ $u(x)=\Sigma_I{u_I(x)}dx_I$ is called a differential $l\hbox{-form}$, if $u_I$ is differentiable. We use $D^{'}{(\Omega,\Lambda^l)}$ to denote the differential $l\hbox{-form}$ space, and $L^s(\Omega,\Lambda^l)$ consists of all $l$-forms $u(x)$ on $\Omega$ satisfying $\int_\Omega|u_I|^s<\infty$. In particular, we know that a $0\hbox{-form}$ is a function.

A differential $l$-form $u\in D'(\Omega,\Lambda^l)$ is called a closed form if $du=0$ in $\Omega$. Similarly, a differential $(l+1)$-form $v\in D'(\Omega,\Lambda^{l+1})$ is called a coclosed form if $d^{\star}v=0$. From the Poincar$\acute{e}$ lemma, $ddu=0$, we know that $du$ is a closed form. The module of a differential form $u$ is given by $|u|^2=\star(u\wedge \star u)\in D'(\Omega,\Lambda^0)$, in other words, it is a function.
The homotopy operator $T:C^{\infty}(\Omega,\Lambda^l)\rightarrow C^{\infty}(\Omega,\Lambda^{l-1})$ is a very important operator in differential form theory, given by
$$Tu=\int_\Omega\psi(y)K_yudy,$$
where $\psi\in C^\infty_0(\Omega)$ is normalized by $\int_\Omega\psi(y)dy=1$, and $K_y$ is a liner operator defined by
$$(K_yu)(x;\xi_1,\cdots,\xi_{l-1})=\int^1_0t^{l-1}u(tx+y-ty;x-y;\xi_1,\cdots,\xi_{l-1})dt.$$
See \cite{8} for more of the function $\psi$ and operator $K_y$. About the homotopy operator $T$, we have the following  decomposition, which will be used repeatedly in this paper,
$$u=d(Tu)+T(du)$$
for any differential form $u\in L^p(\Omega,\Lambda^l),1\leq p<\infty$.
A closed form $u_\Omega$ is defined by $u_\Omega=d(Tu)$, $l=1, \cdots, n$, and when $u$ is a differential $0$-form, $u_\Omega=|\Omega|^{-1}\int_\Omega u(y)dy.$

\vskip 8 pt

The Orlicz space $L^\varphi (\Omega, \mu)$ consists of all
measurable functions $f$ on $\Omega$ such that
$\int_\Omega \varphi\left({|f| \over \lambda} \right) d\mu < \infty$
for some $\lambda=\lambda(f) >0$. $L^\varphi (\Omega, \mu)$ is equipped with the nonlinear Luxemburg functional
$$\|f \|_{\varphi (\Omega, \mu)} = {\hbox{\rm{inf }}} \{\lambda >0: \ \int_\Omega \varphi \left({|f| \over \lambda} \right)
d\mu \leq 1 \},  $$
where the Radon measure $\mu$ is defined by $d\mu = g(x) dx$ and
$g(x) \in A (\alpha, \beta, \gamma; \Omega)$.
A convex Orlicz function $\varphi$ is often called a Young function. If $\varphi$ is a Young function, then
$\| \cdot \|_{\varphi(\Omega, \mu)}$ defines a norm in $L^\varphi (\Omega, \mu)$, which is called the Orlicz norm or Luxemburg norm. Especially, when $\mu$ is the Lebesgue measure, we let $\| \cdot \|_{\varphi(\Omega, \mu)}=\| \cdot \|_{\varphi,\Omega}$ for convenience.

\vskip 8 pt

We say the Young function $\varphi$ belongs to the $G(p,q,c)$-class, $1\leq p <q<\infty,c\geq 1$, if  $\varphi$ satisfies that:
\noindent(1)\ $\frac{1}{c}\leq {\varphi(t^{1/p})}/{g(t)}\leq c$; (2)\ $\frac{1}{c}\leq {\varphi(t^{1/q})}/{h(t)}\leq c$, for every $t>0$, where $g$ is a convex increasing function and $h$ is a concave increasing function on $[0,\infty]$. From \cite{9}, each of $\varphi$, $g$ and $h$ in above definition is doubling in the sense that its values at $t$ and $2t$ are uniformly comparable for all $t>0$, and the consequent fact that
$$c_1t^q\leq h^{-1}(\varphi(t))\leq c_2t^q, c_1t^p\leq g^{-1}(\varphi(t))\leq c_2t^p,$$
where $c_1$ and $c_2$ are constants. Especially, if we choose $\varphi (t) = t^p$, the following estimate for conjugate $A$-harmonic tensors in $\Omega\subset\mathbb{R}^n$ can be established by the similar way in \cite{nolder1999hardy}.
$$\|u\|_{s,\Omega}\leq C  |B|^\beta\|v\|_{t,\Omega},$$
where $\beta=\beta(n,p,q,s,t)$. In Section \ref{yy}, we will give a more general estimate for conjugate $A$-harmonic tensors.
\noindent
Now,  we give the definition of $L^\varphi$-BMO norm.

\begin{definition}\label{d2}
For $u\in L^{1}_{loc}(\Omega,\Lambda^{l}), l=0,1,\cdots,n$, $\varphi$ is a Young function, we write $u\in L^\varphi$-$BMO(\Omega,\Lambda^l)$, if
$$\|u\|_{\varphi*,\Omega}=\sup_{\sigma{B}\subset{\Omega}}|B|^{-1}\|u-u_B\|_{\varphi,B}<\infty$$
for some $\sigma>1$.
\end{definition}

Similarly, we can define the following $L^\varphi$-Lipschitz norm.

\begin{definition}\label{d1}
For $u\in L^{1}_{loc}(\Omega,\Lambda^{l}), l=0,1,\cdots,n$, $\varphi$ is a Young function, we write $u\in L^\varphi$-$Lip_{loc, k}(\Omega,\Lambda^l),0<k<1$, if
$$\|u\|_{\varphi loc Lip_{k},\Omega}=\sup_{\sigma{B}\subset{\Omega}}|B|^{\frac{-(n+k)}{n}}\|u-u_B\|_{\varphi,B}<\infty$$
for some $\sigma>1$.
\end{definition}

The following definition for  the WRH$(\Lambda^{l},\Omega)$-class appears in \cite{Johnson2013Integral}.
\begin{definition}\label{d2} We call $u(x)\in D^{'}(\Omega,\Lambda^{l})$ belongs to the WRH$(\Lambda^{l},\Omega)$-class, $l=0,1,\cdots,n$, if there exists a constant $C>0$ such that $u(x)$ satisfies
$$
\|u\|_{s,B}\leq C |B|^\frac{t-s}{st}\|u\|_{t,\rho B}
$$
for every $0<s,t< \infty$, where all balls $B\subset \Omega$ with $\rho B \subset \Omega$  and  $\rho >1$ is a constant.
\end{definition}

For the upcoming main results, we also need the following  lemmas, given by T. Iwaniec and  A. Lutoborski in \cite{8}.
\begin{lemma} \label{d3}
 Let $u\in L^t(\Omega,\Lambda^{l}),l=1,2,\ldots,n,1<t<\infty$ and $T$ be the  homotopy operator defined on differential forms. Then, there exists a constant $C$, independent of $u$, such that
$$\|Tu\|_{t,\Omega}\leq C|\Omega|\hbox{diam}(\Omega)\|u\|_{t,\Omega}.$$
\end{lemma}

\begin{lemma} \label{d31}
 Let $u\in L^t(\Omega,\Lambda^{l}),l=1,2,\ldots,n,1<t<\infty$. Then, there exists a constant $C$, independent of $u$, such that
$$\|u_\Omega\|_{t,\Omega}\leq C|\Omega|\|u\|_{t,\Omega}.$$
\end{lemma}

\begin{lemma}\label{a26}
 Let $u\in D'(\Omega,\Lambda^l)$ be such that $du\in L^t(\Omega,\Lambda^{l+1})$. Then $u-u_\Omega$ is in $L^{\frac{nt}{n-t}}(\Omega,\Lambda^l)$ and
 $$\left(\int_\Omega|u-u_\Omega|^{\frac{nt}{n-t}}\right)^{\frac{n-t}{nt}}\leq C\left(\int_\Omega|du|^t\right)^{\frac{1}{t}},$$
 where $l=1,2,\ldots,n,1<t<n$.
\end{lemma}

The following lemma appears in \cite{9}.
\begin{lemma}\label{d5}
Take $\psi$ defined on $[0,+\infty)$  be a strictly increasing convex function,  $\psi(0)=0$, and $\Omega\subset \mathbb{R}^{n}$ be a domain. Assume that $u(x)\in D^{'}(\Omega,\Lambda^{l})$ satisfies $\psi(k(|u|+|u_\Omega|))\in L^{1}(\Omega,\mu)$ for any real number $k>0$, and $\mu(x\in\Omega:|\mu-\mu_\Omega|>0)>0$
where $\mu$ be a Radon measure defined by $d\mu(x)=\omega(x)dx$  with a weight $\omega(x)$, then for any $a>0$, we obtain
$$
\int_{\Omega}\psi(a|u|)d\mu \leq C\int_{\Omega}\psi(2a|u-u_\Omega|)d\mu,
$$
where C is a positive constant.
\end{lemma}
\section{Comparison  Theorems for the $L^\varphi$-BMO Norm, $L^\varphi$-Lipschitz Norm and $L^\varphi$ Norm }
In this section, we give two main theorems for the homotopy operator.  Theorem \ref{d6} is the $L^\varphi$-Lipschitz norm inequality for the homotopy operator acting on the differential forms which belong to $ WRH(\Lambda^l,\Omega)$-class. Theorem \ref{d11} is the estimate for  $L^\varphi$-BMO norm with the  exponents $p,q$ in $G(p,q,c)$-class satisfying $q(n-p)< np$.
\begin{theorem}\label{d6}
Let $\varphi$ be a Young function in the $G(p,q,c)$-class, $1\leq p<q<\infty,c\geq1$, $u$ be a differential form such that $u\in WRH(\Lambda^l,\Omega)$-class, $l=1,2,\ldots,n,$ and $\varphi(|u|)\in L^1_{loc}(\Omega)$. Then, there exists a constant C, independent of u, such that
$$\|Tu\|_{\varphi loc\; Lip_k,\Omega}\leq C\|u\|_{\varphi,\Omega},$$
where $\Omega$ is a bounded domain.
\end{theorem}

\noindent
\begin{proof} From the definition of $G(p,q,c)$-class and Jensen's inequality, we obtain
\begin{eqnarray}
\int_B\varphi\left(|u-u_B|\right)dx&=&h\left(h^{-1}\left(\int_B\varphi\left(|u-u_B|\right)dx\right)\right) \nonumber\\
&\leq&h\left(\int_Bh^{-1}\left(\varphi\left(|u-u_B|\right)\right)dx\right) \nonumber\\
&\leq&h\left(C_1\int_B|u-u_B|^qdx\right) \nonumber\\
&\leq&C_2\varphi\left(\left(C_1\int_B|u-u_B|^qdx\right)^{1/q}\right) \nonumber\\
&\leq&C_3\varphi\left(\left(\int_B|u-u_B|^qdx\right)^{1/q}\right)\label{fai1}.
\end{eqnarray}
Replacing $u$ by $Tu$ follows
\begin{eqnarray}\label{3}
\int_B\varphi\left(|Tu-(Tu)_B|\right)dx\leq C_3\varphi\left(\left(\int_B|Tu-(Tu)_B|^qdx\right)^{1/q}\right).
\end{eqnarray}
\noindent
Applying the decomposition theorem of differential form to $Tu$, we have
\begin{eqnarray}\label{9}
Tu=dT(Tu)+Td(Tu).
\end{eqnarray}
Noticing $(Tu)_B=dT(Tu)$, combining \ref{9}, Lemma \ref{d3} and Lemma \ref{d31}, we find
\begin{eqnarray}\label{1}
\left(\int_B|Tu-(Tu)_B|^qdx\right)^{1/q}
&=& \left(\int_B|TdTu|^qdx\right)^{1/q}\nonumber\\
&\leq&C_4(n,q)|B|\hbox{diam}(B) \left(\int_B|dTu|^qdx\right)^{1/q}\nonumber\\
&=&C_4(n,q)|B|\hbox{diam}(B) \left(\int_B|u_B|^qdx\right)^{1/q}\nonumber\\
&\leq&C_5(n,q)|B|^2\hbox{diam}(B) \left(\int_B|u|^qdx\right)^{1/q}.
\end{eqnarray}
Noticing that $u\in WRH(\Lambda^l,\Omega)$-class, so the following inequality holds
\begin{eqnarray}\label{2}
 \left(\int_{ B}|u|^qdx\right)^{1/q}\leq C_6|B|^{(p-q)/pq}\left(\int_{\sigma B}|u|^pdx\right)^{1/p},
 \end{eqnarray}
where $\sigma>1$ is a constant.
Combining (\ref{1}) and (\ref{2}), we have
\begin{eqnarray}
\left(\int_B|Tu-(Tu)_B|^qdx\right)^{1/q}&\leq& C_7|B|^2(\hbox{diam}(B))|B|^{(p-q)/pq}\left(\int_{\sigma B}|u|^pdx\right)^{1/p}\nonumber.
\end{eqnarray}
Noticing that $1<p,q<\infty$, so $1+(p-q)/pq>0$, we derive that
\begin{equation}\label{guocheng1}
  \left(\int_B|Tu-(Tu)_B|^qdx\right)^{1/q}\leq C_8|B|^{1+1/n}\left(\int_{\sigma B}|u|^pdx\right)^{1/p}.
\end{equation}
Since $\varphi$ is an increasing function, using Jensen's inequality and the definition of $G(p,q,c)$-class, we have
\begin{eqnarray}\label{4}
&&
\varphi\left(\left(\int_B|Tu-(Tu)_B|^qdx\right)^{1/q}\right)\cr
&\leq&\varphi\left(C_8|B|^{1+1/n}\left(\int_{\sigma B}|u|^pdx\right)^{1/p}\right)\nonumber\\
&=&\varphi\left(\left(C^p_8|B|^{p(1+1/n)}\int_{\sigma B}|u|^pdx\right)^{1/p}\right)\nonumber\\
&\leq&C_9g\left(C^p_8|B|^{p(1+1/n)}\int_{\sigma B}|u|^pdx\right)\nonumber\\
&=&C_9g\left(\int_{\sigma B}C^p_8|B|^{p(1+1/n)}|u|^pdx\right)\nonumber\\
&\leq&C_9\int_{\sigma B}g\left(C^p_8|B|^{p(1+1/n)}|u|^p\right)dx\nonumber\\
&\leq&C_{10}\int_{\sigma B}\varphi\left(C_8|B|^{1+1/n}|u|\right)dx\nonumber\\
&\leq&C_{11}\int_{\sigma B}\varphi\left(|B|^{1+1/n}|u|\right)dx.
\end{eqnarray}
Combining (\ref{3}) and (\ref{4}) yields that
$$
\int_B\varphi\left(|Tu-(Tu)_B|\right)dx\leq C_{12}\int_{\sigma B}\varphi\left(|B|^{1+1/n}|u|\right)dx.
$$
Noticing that $\varphi$ is doubling, so we obtain
$$\int_B\varphi\left(\frac{|Tu-(Tu)_B|}{\lambda}\right)dx\leq C_{12}\int_{\sigma B}\varphi\left(\frac{|B|^{1+1/n}|u|}{\lambda}\right)dx$$
for  any $\lambda>0$, and from the Orlicz norm definition, we know
\begin{eqnarray}\label{guocheng2}
\|Tu-(Tu)_B\|_{\varphi, B}&\leq& C_{12}\|(|B|^{1+1/n}u)\|_{\varphi, \sigma B}\nonumber\\
&\leq& C_{12}|B|^{1+1/n}\|u\|_{\varphi, \sigma B}
\end{eqnarray}
For all balls $\sigma'B\subset\Omega$ with $\sigma'>\sigma$, we have
\begin{eqnarray}\label{guocheng3}
\|Tu\|_{\varphi loc\; Lip_k,\Omega}&=&\sup_{\sigma'{B}\subset{\Omega}}|B|^{\frac{-(n+k)}{n}}\|Tu-(Tu)_B\|_{\varphi,B}\nonumber\\
&\leq&\sup_{\sigma'{B}\subset{\Omega}}|B|^{\frac{-(n+k)}{n}}C_{12}|B|^{1+1/n}\|u\|_{\varphi,\sigma B}\nonumber\\
&\leq&\sup_{\sigma'{B}\subset{\Omega}}C_{12}|B|^{1+\frac{1}{n}+\frac{-(n+k)}{n}}\|u\|_{\varphi,\sigma B}.
\end{eqnarray}

\noindent
As $1+\frac{1}{n}+\frac{-(n+k)}{n}>0$, so we have
$$\|Tu\|_{\varphi loc\; Lip_k,\Omega}\leq C\|u\|_{\varphi,\Omega}.$$
\end{proof}

If we assume the Lebesgue measure $|\{x\in B:|u-u_B|>0\}|>0$, using Lemma \ref{d5} with $\psi(t)=\varphi (t)$, $\omega(x)=1$ over the ball $B$, we have the following corollary.
\begin{corollary}\label{d7}
Let $\varphi$ be a Young function in the $G(p,q,c)$-class, $1\leq p<q<\infty, c\geq1$, $u$ be a differential form such that $u\in WRH(\Lambda^l,\Omega)$-class, $l=1,2,\ldots,n$, $|\{x\in B:|u-u_B|>0\}|>0$, and $\varphi(|u|)\in L^1_{loc}(\Omega)$. Then, there exists a constant C, independent of u, such that
$$
\|Tu\|_{\varphi loc\; Lip_k,\Omega}\leq C\|u\|_{\varphi*,\Omega},
$$
 where $\Omega$ is a bounded domain.
\end{corollary}

\begin{theorem}\label{d11}
Let $\varphi$ be a Young function in  $G(p,q,c)$-class, $1<p<q<\infty$, $c\geq1,q(n-p)< np$, and
$u\in L^p(\Omega,\Lambda^l),l=1,2,\ldots,n,$ be a differential form such that $\varphi(|u|)\in L^1_{loc}(\Omega)$. Then, there exists a constant C, independent of u, such that
$$\|Tu\|_{\varphi *,\Omega}\leq C\|u\|_{\varphi,\Omega},$$
where $\Omega$ is a bounded domain.
\end{theorem}

\noindent
\begin{proof} For the case that $1<p<n$, $q(n-p)< np$ means $q<\frac{np}{n-p}$. So using the monotonic
property of the $L^p$ space, Lemma \ref{a26} and Lemma \ref{d31}, for any differential form $u\in L^p(\Omega,\Lambda^l)$, we have
\begin{eqnarray}\label{2+}
\left(\int_B|Tu-(Tu)_B|^qdx\right)^{1/q}&\leq&|B|^{\frac{1}{q}-\frac{1}{p}+\frac{1}{n}}\left(\int_B|Tu-(Tu)_B|^\frac{np}{n-p}dx\right)^{\frac{n-p}{np}}\nonumber\\
&\leq&C_1|B|^{\frac{1}{q}-\frac{1}{p}+\frac{1}{n}}\left(\int_B|dTu|^p dx\right)^{\frac{1}{p}}\nonumber\\
&=&C_1|B|^{\frac{1}{q}-\frac{1}{p}+\frac{1}{n}}\left(\int_B|u_B|^p dx\right)^{\frac{1}{p}}\nonumber\\
&\leq&C_2|B|^{\frac{1}{q}-\frac{1}{p}+\frac{1}{n}+1}\left(\int_B|u|^p dx \right)^{\frac{1}{p}}.\nonumber\\
\end{eqnarray}
Next, for the case that $n\leq p<q<\infty$, we can choose $s$ with $1<s<n$ such that $q<\frac{ns}{n-s}$ (remark: it is possible from that $\frac{ns}{n-s}\rightarrow\infty$, as $s\rightarrow n$). Thus, applying Lemma \ref{a26} and Lemma \ref{d31} and noticing the monotonic
property of the $L^p$ space with $s<p$, we have
\begin{eqnarray}\label{5}
  \left(\int_B|Tu-(Tu)_B|^\frac{ns}{n-s}dx\right)^{\frac{n-s}{ns}}&\leq& C_{1'}\left(\int_B|dTu|^s dx\right)^{\frac{1}{s}} \cr
 &=& C_{1'}\left(\int_B|u_B|^s dx\right)^{\frac{1}{s}}\cr
   &\leq& C_{2'}\left(\int_B|u|^s dx\right)^{\frac{1}{s}} \cr
  &\leq& C_{2'}|B|^{\frac{1}{s}-\frac{1}{p}}\left(\int_B|u|^p dx\right)^{\frac{1}{p}}.
\end{eqnarray}
Combining the monotonic property of the $L^p$ space with $q<\frac{ns}{n-s}$ and (\ref{5}) yields
\begin{eqnarray}\label{6}
\left(\int_B|Tu-(Tu)_B|^qdx\right)^{1/q}&\leq&|B|^{\frac{1}{q}-\frac{1}{s}+\frac{1}{n}}\left(\int_B|Tu-(Tu)_B|^\frac{ns}{n-s}dx\right)^{\frac{n-s}{ns}}\cr
&\leq& C_{3'}|B|^{\frac{1}{q}-\frac{1}{s}+\frac{1}{n}}|B|^{\frac{1}{s}-\frac{1}{p}}\left(\int_B|u|^p dx\right)^{\frac{1}{p}}\cr
&=& C_{3'}|B|^{\frac{1}{q}-\frac{1}{p}+\frac{1}{n}}\left(\int_B|u|^p dx\right)^{\frac{1}{p}}\cr
&\leq&C_{3'}|B|^{\frac{1}{q}-\frac{1}{p}+\frac{1}{n}+1}\left(\int_B|u|^p dx\right)^{\frac{1}{p}}.
\end{eqnarray}
Since ${\frac{1}{q}-\frac{1}{p}+\frac{1}{n}}>0$, the inequalities (\ref{2+}) and (\ref{6}) indicate that
\begin{eqnarray}\label{7}
\left(\int_B|Tu-(Tu)_B|^qdx\right)^{1/q}&\leq&C_{3}|B|\left(\int_B|u|^p dx\right)^{\frac{1}{p}}
\end{eqnarray}
holds for all $1<p<q<\infty$ with $q(n-p)< np$.

\noindent
Now, beginning with (\ref{7}) and using the similar process from inequality (\ref{guocheng1}) to inequality (\ref{guocheng2}), we get
\begin{eqnarray}\label{8}
\|Tu-(Tu)_B\|_{\varphi, B}&\leq& C_{4}\||B|u\|_{\varphi, B}\nonumber\\
&\leq& C_{4}|B|\|u\|_{\varphi, B}.\nonumber\\
\end{eqnarray}

\noindent
According to the definition of the $L^{\varphi}$-BMO norm and (\ref{8}), we obtain
\begin{eqnarray}
\|Tu\|_{\varphi *,\Omega}&=&\sup_{\sigma{B}\subset{\Omega}}|B|^{-1}\|Tu-(Tu)_B\|_{\varphi,B}\nonumber\\
&\leq&\sup_{\sigma{B}\subset{\Omega}}|B|^{-1}C_{4}|B|^{1}\|u\|_{\varphi, B}\nonumber\\
&=&\sup_{\sigma{B}\subset{\Omega}}C_{4}\|u\|_{\varphi, B}\nonumber\\
&\leq&C\|u\|_{\varphi,\Omega}.
\end{eqnarray}
\end{proof}

\noindent
Remark 1: The differential form $u$ in Theorem \ref{d11} does not need satisfy the conditions of $WRH(\Lambda^l,\Omega)$-class in Theorem \ref{d6}. But, we restrain the exponents in $G(p,q,c)$-class.

Now we compare the $L^\varphi$-Lipschitz norm and $L^\varphi$-BMO norm of differential forms.

\begin{theorem}\label{d8}
 Let $\varphi$ be a Young function, $u\in D'(\Omega,\Lambda^{l}),l=1,2,\cdots,n,$  be a differential form in $\Omega$ and $\varphi(|u|)\in L^1_{loc}(\Omega,x).$ Then, there exists a constant $C$, independent of $u$, such that
$$\|u\|_{\varphi*,\Omega}\leq{C}\|u\|_{\varphi loc\; Lip_k,\Omega},$$
where $k$ is a constant with $0<{k}<1$.
\end{theorem}

\vspace{3mm}
\noindent
\begin{proof} From the definition of BMO norm, we have
\begin{eqnarray*}
\|u\|_{\varphi*,\Omega} &=& \sup_{\sigma{B}\subset\Omega }|B|^{-1}\|u-u_B\|_{\varphi,B} \\
 &=& \sup_{\sigma{B}\subset\Omega }|B|^{k/n}|B|^{-(n+k)/n}\|u-u_B\|_{\varphi,B} \\
 &\leq & \sup_{\sigma{B}\subset\Omega }|\Omega|^{k/n}|B|^{-(n+k)/n}\|u-u_B\|_{\varphi,B} \\
 & \leq & |\Omega|^{k/n}\sup_{\sigma{B}\subset\Omega }|B|^{-(n+k)/n} \|u-u_B\|_{\varphi,B} \\
 & \leq &{C}\sup_{\sigma{B}\subset\Omega }|B|^{-(n+k)/n} \|u-u_B\|_{\varphi,B} \\
 & \leq & {C}\|u\|_{ \varphi loc\; Lip_k,\Omega}.
\end{eqnarray*}
\end{proof}

Replacing $u$ by $Tu$, and combining Theorem \ref{d6}, we obtain the following corollary.
\begin{corollary}\label{d9}
Let $\varphi$ be a Young function in the class $G(p,q,c),1\leq p<q<\infty,c\geq1$, $u$ be a differential form such that $u\in WRH(\Lambda^l,\Omega)$-class,$l=1,2,\ldots,n,$ and $\varphi(|u|)\in L^1_{loc}(\Omega)$, where $\Omega$ is a bounded domain. Then, there exists a constant C, independent of u, such that
$$\|Tu\|_{\varphi *,\Omega}\leq C\|u\|_{\varphi,\Omega}.$$
\end{corollary}
\section{Applications}\label{yy}

We call $u$ and $v$ a pair of
conjugate $A$-harmonic tensors in $\Omega$, if $u$ and $v$ satisfy
the conjugate $A$-harmonic equation
\begin{equation}\label{4.1}
  A(x, d u) =d^{\star} v,
\end{equation}
where $ A : \Omega \times \Lambda^l$
($\mathbb{R}^n$) $\to  \Lambda^l$($\mathbb{R}^n$) is invertible and satisfies the following
conditions:
\begin{equation}\label{4.2}
  |A(x, \xi)| \leq a|\xi|^{q-1} \ \ \  \hbox{and} \ \ \  <A(x, \xi),
\xi>\ \ \geq \ |\xi|^q
\end{equation}
 for almost every $x \in
\Omega$ and all $\xi \in \Lambda^l$ ($\mathbb{R}^n$). \rm
Here, $ a>0$ is
a constant and $1<q< \infty$ is a fixed exponent associated with
(\ref{4.1}).
In  recent years, the results for conjugate $A$-harmonic tensors are widely used in quasiregular mappings, and the theory of elasticity. In 1999, the following inequality for conjugate $A$-harmonic tensors in $\Omega$ was given by Nolder in \cite{nolder1999hardy},
$$\|u\|^q_{loc lip_k ,\Omega}\leq C \|v\|^p_{loc lip_k,\Omega},$$
where $0<l,k<1 $ satisfies $q(k-1)=p(l-1)$. Now, we give the $L^\varphi$-BMO norm estimate for conjugate $A$-harmonic tensors in $\Omega$.
\begin{theorem}\label{3323}
Let $\varphi$ be a Young function in the class $G(p,q,c)$ with $1\leq p<q<\infty,\frac{1}{p}+\frac{1}{q}=1$ and $c\geq1$. $u$ and $v$ are conjugate $A$-harmonic tensors  such that  $\varphi(|v|)\in L^1_{loc}(\Omega)$. The fixed exponent associated with conjugate $A$-harmonic equation is $q$. Then, there exists a constant C, independent of $u$ and $v$, such that
$$\|u\|_{\varphi *,\Omega}\leq C  |B|^\beta\|v\|_{\varphi *,\Omega},$$
where $\beta=1+\frac{1}{n}-\frac{p}{nq}$ and $\Omega$ is a bounded domain.
\end{theorem}

\noindent
\begin{proof} From the inequality (\ref{fai1}) in Theorem \ref{d6}  and Lemma \ref{d3}, we have
\begin{eqnarray}\label{2.11}
\int_B\varphi\left(|u-u_B|\right)dx
&\leq&C_1\varphi\left(\left(\int_B|u-u_B|^qdx\right)^{1/q}\right)\cr
&\leq&C_1\varphi\left(C_2|B|^{1+\frac{1}{n}}\left (\int_B|du|^qdx\right)^{1/q}\right).\nonumber
\end{eqnarray}
Using the inequality $|du|^q\leq |\ast dv|^p$(which appears in Theorem $3.1$ in \cite{nolder1999hardy}), we obtain
\begin{eqnarray}\label{2.12}
&&\varphi\left(C_2|B|^{1+\frac{1}{n}}\left (\int_B|du|^qdx\right)^{1/q}\right)
\cr
&\leq&\varphi\left(C_2|B|^{1+\frac{1}{n}}\left (\int_B|\ast dv|^pdx\right)^{1/q}\right)\cr
&=&\varphi\left(C_2|B|^{1+\frac{1}{n}}\| d\ast v\|_{p,B}^{p/q}\right)\cr
&\leq&\varphi\left(C_2|B|^{1+\frac{1}{n}}\left(C_3|B|^{\frac{1}{n}}\|\ast v-\ast \theta\|_{p,\rho B}\right)^{p/q}\right)\cr
&\leq&\varphi\left(C_4|B|^{1+\frac{1}{n}-\frac{p}{nq}}\left(\int_{\rho B}|\ast v-\ast \theta|^pdx\right)^{1/q}\right)\cr
&\leq&C_5g\left(C_4^q|B|^{q+\frac{q}{n}-\frac{p}{n}}\left(\int_{\rho B}|\ast v-\ast \theta|^pdx\right)\right)\cr
&\leq&C_5C_4^q|B|^{q+\frac{q}{n}-\frac{p}{n}}\int_{\rho B}g\left(|\ast v-\ast \theta|^p\right)dx,
\end{eqnarray}
where $\theta$ is any closed form, and the third inequality is from the Caccioppoli inequality for conjugate $A$-harmonic tensors. The properties of $G(p,q,c)$-class yields
\begin{eqnarray}
\int_{\rho B}g\left(|\ast v-\ast \theta|^p\right)dx
&\leq& C_6\int_{\rho B}\varphi\left(| v-\theta|\right)dx.
\end{eqnarray}
Choose $\theta=v_B$, and similar to the proof of inequalities (\ref{guocheng2}) and (\ref{guocheng3}) in Theorem \ref{d6}, we have
$$
  \|u\|_{\varphi *,\Omega}\leq C  |B|^\beta\|v\|_{\varphi *,\Omega}.
$$
\end{proof}
\vskip 8 pt

\vskip 8 pt

Next, we give a weighted estimate for differential forms. The weight we choose is called  $A(\alpha,\beta,\gamma,\Omega)$ weight which satisfies $\omega(x)>0$ a.e., and $$\sup_{B\subset \Omega}\left({{1}\over{|B|}}\int_B\omega^\alpha\,dx\right)\left({{1}\over{|B|}}\int_B\omega^{-\beta}\,dx\right)^{\gamma/\beta}<\infty$$
for some positive constants $\alpha,\beta,\gamma$. One may readily see that the well-known $A_p$ weight is a special $A(\alpha,\beta,\gamma,\Omega)$ weight, more properties for $A(\alpha,\beta,\gamma,\Omega)$ weight see \cite{10}. We need the following  lemma for Orlicz functions.
\begin{lemma}\label{iit}
Let $\varphi$ be a Young function such that $\varphi(x)\leq x^p$ for any $x>0$, $u\in L^p(\Omega,\Lambda^{l}),l=1,2,\cdots,n,$  be a differential form in $\Omega$. Then, for any $\omega\in A(\alpha,\beta,\gamma,\Omega)$, we have
$$\|u\|_{\varphi
,\omega ,B}\leq{C}\|u\|_{p,\omega,B},$$
 where $C$ is a constant independent of $u$.
\end{lemma}

\noindent
\begin{proof} Young function $\varphi\geq 0$ gives

\begin{eqnarray*}
\int_B\varphi\left(\frac{|u(x)|}{\|u(x)\|_{p,\omega,B}}\right)\omega(x)dx&\leq&\int_B\left(\frac{|u(x)|}{\|u(x)\|_{p,\omega,B}}\right)^p\omega(x)dx\\
&=&\frac{\int_B|u(x)|^p\omega(x)dx}{\|u(x)\|^p_{p,\omega,B}}\\
&=&1.
\end{eqnarray*}
Then, according to the definition of $L^{\varphi}$-norm, it implies that
$$
\inf\left\{\lambda>0:\int_B\varphi\left(\frac{|u(x)|}{\lambda}\right)\omega(x)dx\leq 1\right\}\leq \|u(x)\|_{p,\omega,B}.
$$
That is
$$\|u\|_{\varphi,\omega,B}\leq\|u\|_{p,\omega,B}.$$
\end{proof}
\begin{theorem}\label{T4.3}
 Let $\varphi$ be a Young function such that $\varphi(x)\leq x^s$, $u\in{L^{p}(\Omega,\Lambda^{l},\mu)}, l=1, 2, \cdots, n$,  be a differential form in $\Omega$,  Radon measure $\mu$ is defined by $\omega(x)dx=d\mu$, and $\omega(x)\in{A(\alpha,\beta,\gamma,\Omega)}$ for some $\alpha>1$,   $\beta=\frac{\alpha{q}}{\alpha{p}-p-\alpha{q}}$, $\gamma=\frac{\alpha{q}}{p}$, and $\alpha{p}-p-\alpha{q}>0,$ where $1\leq s<q<\infty$. Then, there exists a constant $C$, independent of $u$, such that
$$\|u\|_{\varphi loc\; Lip_k,\omega,\Omega}\leq{C}\|u\|_{p,\omega,\Omega},$$
where $k$ is a constant with $0<{k}<1.$
\end{theorem}
\begin{proof}  Applying Lemma \ref{iit}, we have
\begin{eqnarray*}
\|u(x)-u_B\|_{\varphi,\omega,B}&\leq &\|u(x)-u_B\|_{s,\omega,B}\cr
&\leq& |B|^{\frac{1}{s}-\frac{1}{q}}\|u(x)-u_B\|_{q,\omega,B} \cr
&= &|B|^{\frac{1}{s}-\frac{1}{q}}\left(\int_B|u(x)-u_B|^q\omega(x)dx\right)^\frac{1}{q}.
\end{eqnarray*}
Similar to the proof of Theorem 4.2 in \cite{Li2015Lipschitz}, using the H\"{o}lder inequality to get
\begin{eqnarray*}
&&\left(\int_B|u(x)-u_B|^q\omega(x)dx\right)^\frac{1}{q}\cr
&\leq& \left(\int_B|u(x)-u_B|^\frac{\alpha q}{\alpha-1}dx\right)^\frac{\alpha-1}{\alpha q}\left(\int_B\omega(x)^\alpha dx\right)^\frac{1}{\alpha q}\cr
&\leq&  C_1|B|^{1+\frac{1}{n}}\left(\int_B|u(x)|^\frac{\alpha q}{\alpha-1}dx\right)^\frac{\alpha-1}{\alpha q}\left(\int_B\omega(x)^\alpha dx\right)^\frac{1}{\alpha q}\cr
&\leq&  C_1|B|^{1+\frac{1}{n}}\left(\int_B|u(x)|^p\omega(x)dx\right)^\frac{1}{p}\cr
&&\ \ \ \ \ \ \ \ \ \ \ \ \
\times \left(\int_B\omega(x)^\alpha dx\right)^\frac{1}{\alpha q}\left(\int_B(\omega(x)^{-1})^\frac{\alpha q}{\alpha p-p-\alpha q}dx\right)^\frac{\alpha p-p-\alpha q}{\alpha pq}\cr
&\leq& C_2|B|^{1+\frac{1}{n}} \|u(x)\|_{p,B,\Omega}.
\end{eqnarray*}
The boundedness of the second part in the penultimate inequality above is due to  that $\omega(x)\in{A(\alpha,\beta,\gamma,\Omega)}$. Finally, noting that $\frac{1}{n}+\frac{1}{s}-\frac{1}{q}-\frac{k}{n}>0,$ for $0\leq k\leq1,$ and combining the definition of $L^\varphi$-Lipschitz norm, we can complete the proof of Theorem \ref{T4.3}.
\end{proof}

\noindent
Remark 2: Note that the $A(\alpha, \beta, \gamma; \Omega)$-class is an extension of several existing weight classes which contain $A^\lambda_r(\Omega)$-weight, $A_r(\lambda, \Omega)$-weight and $A_r(\Omega)$-weight. Thus, these conclusions obtained in this paper will change into the corresponding versions when we take some weight as a special case.

\vspace{3mm}
\noindent
{\bf Conflict of Interests}
The authors declare that there is no conflict of interests regarding the publication of this
article.

\vspace{3mm}

\noindent
{\bf Authors' Contributions}
All authors put their efforts together on the research and writing of this manuscript.  Xuexin Li
carried out the proofs of all research results in this manuscript, and wrote its draft. Yuming Xing and Jinling Niu proposed the study, participated in its design and revised its final version.  All authors
read and approved the final manuscript.
\bibliographystyle{amsplain}

\end{document}